\documentclass[letterpaper, 10pt, dvipsnames]{article}

\usepackage{type1ec}
\usepackage[T1]{fontenc}
\usepackage[utf8]{inputenc}    
\usepackage[english]{babel}
\usepackage[normalem]{ulem}
\usepackage{dsfont, bm, pifont, mathrsfs}
\usepackage{stmaryrd}
\usepackage{bbm}
\usepackage{enumitem}
\usepackage{bold-extra}
\usepackage{float, graphicx, wrapfig, tikz}
\usetikzlibrary{positioning, fit}
\usepackage[margin=0.8in]{subcaption}
\usepackage[font=small, margin=0.6in]{caption}

\usepackage{amsthm, amsmath, amsfonts, amssymb, mathrsfs, mathtools}
\usepackage[alphabetic]{amsrefs}
\usepackage{hyperref, xcolor, titlesec} 
\hypersetup{colorlinks = true, urlcolor = RoyalBlue!70!Black,linkcolor=BrickRed, citecolor=RoyalBlue, bookmarksopen = true}
\usepackage[nomarginpar]{geometry}
\numberwithin{equation}{section}
\usepackage[boxsize = .3em]{ytableau}
\usepackage{cleveref}
\DeclareMathOperator{\Tr}{Tr}

\newcommand{\I}{\mathrm{I}}
\newcommand{\II}{\mathrm{II}}
\newcommand{\III}{\mathrm{III}}
\newcommand{\IV}{\mathrm{IV}}
\newcommand{\N}{\mathbb{N}}

\newcommand{\Rbb}{\mathbb{R}}

\let\d\relax
\newcommand{\d}{\mathrm{d}}

\newcommand{\A}{\mathcal{A}}

\newcommand{\B}{\mathcal{B}}

\newcommand{\F}{\mathcal{F}}
\newcommand{\M}{\mathcal{M}}

\newcommand{\R}{\mathcal{R}}

\newcommand{\T}{\mathcal{T}}

\newcommand{\G}{\mathcal{G}}

\newcommand{\W}{\mathcal{W}}

\renewcommand{\u}{\mathbf{u}}
\renewcommand{\v}{\mathbf{v}}
\newcommand{\w}{\mathbf{w}}
\newcommand{\x}{\mathbf{x}}

\newcommand{\db}{\mathbf{d}}

\newcommand{\X}[1]{X^{(#1)}}
\renewcommand{\epsilon}{\varepsilon}
\renewcommand{\phi}{\varphi}

\let\O\relax
\newcommand{\O}[1]{\mathcal{O}\left(#1\right)}

\newcommand{\E}{\mathds{E}}
\newcommand{\unn}[2]{\left[\!\left[#1,#2\right]\!\right]}

\newtheorem{ccounter}{ccounter}[section]
\newtheorem{theorem}[ccounter]{Theorem}
\newtheorem{lemma}[ccounter]{Lemma}

\newtheorem{definition}[ccounter]{Definition}
\newtheorem{proposition}[ccounter]{Proposition}
\newtheorem{assumption}[ccounter]{Assumption}

\theoremstyle{definition}

\titleformat{\section}[block]{\normalfont\filcenter}{\large\Roman{section}.}{.7em}{\large\scshape}

\titleformat{\subsection}[runin]{\normalfont}{\large \bf \thesubsection .}{.5em}{\bf}[.]
\titleformat{\subsubsection}[runin]{\normalfont}{\bf \thesubsubsection .}{.5em}{\bf}[.]
\titleformat*{\paragraph}{\itshape\mdseries}

\begin{document}
\tikzset{every node/.style={circle, minimum size=.1cm, inner sep = 2pt, scale=1}}

\title{\vspace{-5ex}\bfseries\scshape{Hadamard product of independent random sample covariance matrices with correlation structure}}
\author{L. \textsc{Benigni}\thanks{Supported by NSERC RGPIN 2023-03882 \& DGECR 2023-00076 and FRQNT Etablissement de la Rel\`eve Professorale 364387}\\\vspace{-0.15cm}\footnotesize{\it{Universit\'e de Montr\'eal}}\\\footnotesize{\it{lucas.benigni@umontreal.ca}}\and Z. \textsc{Zaklani}\\\vspace{-0.15cm}\footnotesize{\it{Universit\'e de Montr\'eal}}\\\footnotesize{\it{ziyad.zaklani@umontreal.ca}}}
\date{}
\maketitle
\abstract{We compute the asymptotic empirical eigenvalue distribution of the matrix 
\(
M = \bigodot_{i=1}^k \frac{1}{d_i}\X{i}{\X{i}}^\top
\)}
where $X^{(i)}$ are independent matrices with independent rows but general correlation within each row under the dimension scaling $\frac{n}{d_1\dots d_k}\to \gamma$.
\section{Introduction and main results} 

Random matrix theory traces its origins to the seminal paper of Wishart~\cite{wishart1928generalised}, where certain covariance-type matrices were investigated in a statistical context. The field later experienced tremendous growth, propelled in particular by Wigner’s foundational work on spectra of large symmetric matrices~\cite{wigner} and by the contribution of Marchenko and Pastur~\cite{marchenko} in high-dimensional statistics. In the latter, the authors analyzed the empirical spectral distribution of matrices of the form  
\(
\frac{1}{d} XX^\top ,
\)
with $X\in\mathbb R^{n\times d}$ having i.i.d.\ centered entries of variance one (their result in fact allow nontrivial covariance structures among the entries). They established that when $n,d\to\infty$ with $\frac{n}{d}\to\gamma$, the empirical measure  
\(
\frac{1}{n}\sum_{i=1}^n \delta_{\lambda_i}
\)
converges in probability, in the weak sense, to the Marchenko--Pastur law $\mu_{\mathrm{MP}_\gamma}$. For $\gamma\in(0,1]$, this distribution has density  
\begin{equation}\label{eq:mp}
\mu_{\mathrm{MP}_\gamma}(\mathrm dx) = 
\frac{\sqrt{(b-x)(x-a)}}{2\pi \gamma x}\,\mathbf 1_{[a,b]}(x)\,\mathrm dx,
\qquad 
a=(1-\sqrt\gamma)^2,\quad b=(1+\sqrt\gamma)^2 .
\end{equation}
If instead $\gamma>1$, the spectrum of $XX^\top$ includes an additional point mass $(1-1/\gamma)\delta_0$, reflecting the mismatch of dimensions between $XX^\top$ and $X^\top X$. The Marchenko--Pastur theorem has since inspired a vast literature and numerous generalizations under weaker structural assumptions \cites{yin1986limiting, silverstein1995empirical, silverstein1995strong, hachem2007deterministic, baizhou}; see the monograph~\cite{baisilver} for an overview.

In recent years, developments in machine learning have motivated new classes of random matrix models, often exhibiting nonlinear transformations or unconventional scaling regimes; see~\cite{rmmlbook}. A notable direction concerns matrices obtained by applying a nonlinear activation function entrywise to high-dimensional random features, a line of work initiated in the study of kernel methods~\cites{elkaroui, chengsinger} and expanded in the context of large neural networks~\cites{louartliaocouillet, pennington2017nonlinear, benignipeche1, fan2020spectra, piccolo2021analysis, benignipeche2, wangzhu, dabo2024traffic, benignipaquette}. Other structural models have appeared as well. For example, the Neural Tangent Kernel (NTK)~\cites{jacot2018neural, chizat2019lazy}, the covariance matrix of the Jacobian of a network’s outputs during training, can be expressed, in the two-layer case, as  
\begin{equation}\label{eq:NTK}
\mathrm{NTK}
= \frac1d XX^\top \odot \frac1p \sigma'(XW) D^2 \sigma'(XW)^\top 
\;+\; \frac1p \sigma(XW)\sigma(XW)^\top ,
\end{equation}
with $X\in\mathbb R^{n\times d}$ the data matrix, $W\in\mathbb R^{d\times p}$ the random weights at initialization, $D$ the diagonal matrix of output-layer parameters, and $\sigma$ the activation. The spectral behavior of this matrix has been analyzed in~\cites{fan2020spectra, wangzhu} in regimes where the first (Hadamard-product) term is asymptotically trivial, while the regime in which this component persists requires the scaling $n\asymp dp$ and was studied in~\cite{benignipaquette}. Such multivariate polynomial scalings have become increasingly relevant as modern data sets grow extremely large, prompting a wave of recent works~\cites{misiakiewicz2022spectrum, lu2022equivalence, dubova2023universality, hu2024asymptotics, montanarizhong, pandit2024universality, misiakiewicz2023six, kogan2024extremal}.

Despite this surge of activity, comparatively little is known about the spectral properties of Hadamard products of random matrices, even in rather simple settings. The interested reader can see the analysis of entrywise product of structured random matrices \cites{bose2014bulk, mukherjee2023convergence} for instance. In this present article, we consider the following model of matrices. For a family of independent random matrices $\X{i}\in \Rbb^{n\times d_i}$ and a fixed $k\in \N$, we define the matrix 
\[
M = \bigodot_{i=1}^k \frac{1}{d_i}\X{i}{\X{i}}^\top
\quad\text{where}\quad 
(A\odot B)_{ij} = A_{ij}B_{ij}.
\]
We have the following assumption on the entries of $\X{i}$.
\begin{assumption}
    For $i\in[k]$, we suppose that the rows of $\X{i}$ are independent and if $\x^{(i)}_p$ is the $p$-th row of $\X{i}$ we suppose that for all $q\in\unn{1}{d_i}$ we have 
    \(
    \E[x^{(i)}_{pq}]=0
    \) and that there exists a symmetric positive semi-definite matrix $\Sigma^{(i)}\in \Rbb^{d_i\times d_i},$ $C_i>0$ with $\Vert \Sigma^{(i)}\Vert\leqslant C_i,$ such that for any $p\in[n]$
    \[
    \E\left[
        x_{pq}^{(i)}x_{pr}^{(i)}
    \right]
    =
    \sigma_{qr}^{(i)}
    \quad\text{and}\quad 
    \kappa_4^{(i)}\coloneqq\sup_{q,r}\sum_{s,t=1}^{d_i}\left\vert\kappa\left(x_{pq}^{(i)},x_{pr}^{(i)},x_{ps}^{(i)},x_{pt}^{(i)}\right)\right\vert<\infty.
    \]
    We suppose that if $(\lambda^{(i)}_j)_{1\leqslant j\leqslant d_i}$ are the eigenvalues of $\Sigma^{(i)}$ then there exists $\mu_i$ such that 
    \[
    \frac{1}{d_i}\sum_{j=1}^{d_i}\lambda_j^{(i)}\xrightarrow[d_i\to\infty]{}\mu_i\quad\text{weakly}
    \]
    and we denote $C=\max_{1\leqslant i\leqslant k}C_i$ and $\kappa_4 = \max_{1\leqslant i\leqslant k}\kappa_4^{(i)}$.
\end{assumption}
The case $k=2$ with $\Sigma^{(1)}=\mathrm{Id}_{d_1}$ and $\Sigma^{(2)}=\mathrm{Id}_{d_2}$ has been studied in \cite{abou2025eigenvalue} and the asymptotic eigenvalue distribution is simply given by the Marchenko--Pastur distribution with parameter shape $\lim_{n\to\infty} \frac{n}{d_1d_2}.$ We consider a more general Hadamard product with any finite number of matrices with the following assumptions on the scaling of dimensions.
\begin{assumption}
    We suppose that there exists $\gamma>0$ such that 
    \[
    \frac{n}{\prod_{i=1}^k d_i}\xrightarrow[n\to\infty]{}\gamma
    \quad\text{and that}\quad
    d_i\xrightarrow[n\to\infty]{}\infty\text{ for all }i\in [k]. 
    \]
\end{assumption}
We define the following operation on probability measures.
\begin{definition}
    If $\mu$ and $\nu$ are two probability measures, we define $\mu \circledast \nu$ as the classical multiplicative convolution, i.e., the pushforward of $\mu\otimes \nu$ under the multiplication map $m(x, y)=xy.$ 
\end{definition}
\noindent The main reason for this definition is that we have 
\[
\frac{1}{d_id_j}\sum_{p=1}^{d_i}\sum_{q=1}^{d_j}\delta_{\lambda^{(i)}_p\lambda^{(j)}_q}
\xrightarrow[n\to\infty]{}\mu_i \circledast \mu_j.
\]

We also recall the definition of the free multiplicative convolution with the Marchenko--Pastur distribution, also called the Marchenko--Pastur map.
\begin{definition}
    Suppose that $A\succeq 0$ and $B\succeq 0$ where $A\in\Rbb^{a\times a}$, $B\in\Rbb^{b\times b}$ and that $Y\in\Rbb^{a\times b}$ is a random matrix with i.i.d standard entries with all finite moments. If $a\to\infty$ and $b\to\infty$ with $\frac{a}{b}\to\gamma\in(0,\infty)$ and if 
    \[
    \mathrm{limspec}(A)=\mathrm{limspec}(B)=\nu
    \]
    then we have that almost surely
    \[
    \mathrm{limspec}\left(
        \frac{1}{b}A^{\frac{1}{2}}YY^\top A^{\frac{1}{2}}
    \right)
    =
    \mathrm{limspec}\left(
        \frac{1}{a}YBY^\top
    \right)
    =
    \mu_{\mathrm{MP}}^\gamma \boxtimes \nu.
    \]
    It is also characterized by the self-consistent equation followed by its Stieltjes transform 
    \[
    s(z) \coloneqq \int_\Rbb \frac{(\mu_{\mathrm{MP}}^\gamma\boxtimes \nu)(\d x)}{x-z}
    =
    \int_\Rbb \frac{\nu(\d t)}{t(1-\gamma(1+zs(z)))-z}.
    \]
\end{definition}
\begin{theorem}\label{theo:main}
    The empirical eigenvalue distribution of $M$ converges weakly almost surely to $\mu_{\mathrm{MP}}^\gamma \boxtimes \left(\mu_1 \circledast\dots\circledast\mu_k\right)$.
\end{theorem}
We show some examples in Figure \ref{fig:simulation}.

\begin{figure}[ht!]
    \centering
    
    % Row 1
    \begin{subfigure}{0.45\textwidth}
        \centering\captionsetup{width=1.2\linewidth}
        \includegraphics[width=.8\linewidth]{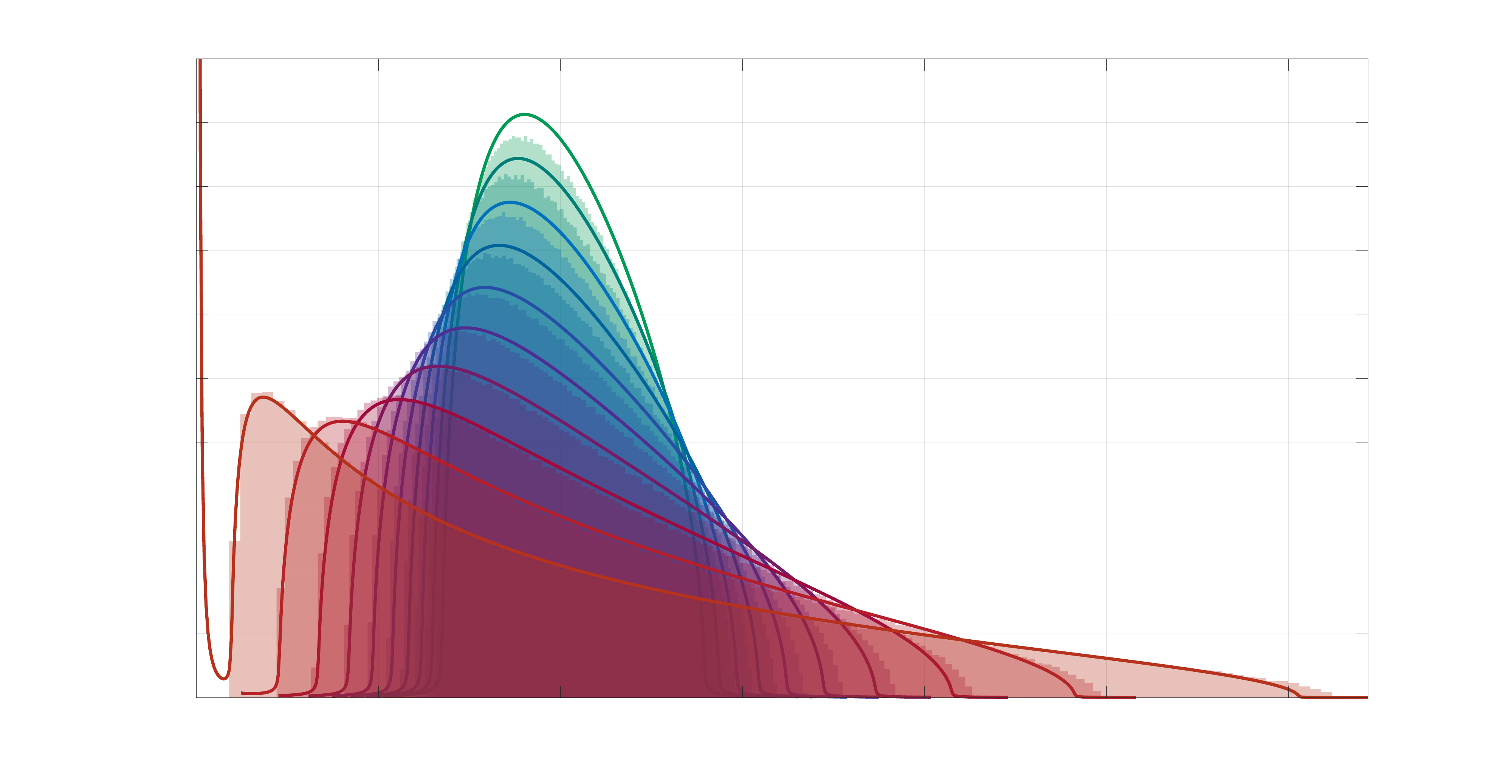}
        \caption{$\mu_1 = \delta_1$, $\mu_2 = \frac{1}{3}\left(\delta_1+\delta_2+\delta_3 \right)$, $\gamma\in[1.5,6].$}
        \label{fig:1}
    \end{subfigure}
    \begin{subfigure}{0.45\textwidth}
        \centering\captionsetup{width=1.2\linewidth}
        \includegraphics[width=.8\linewidth]{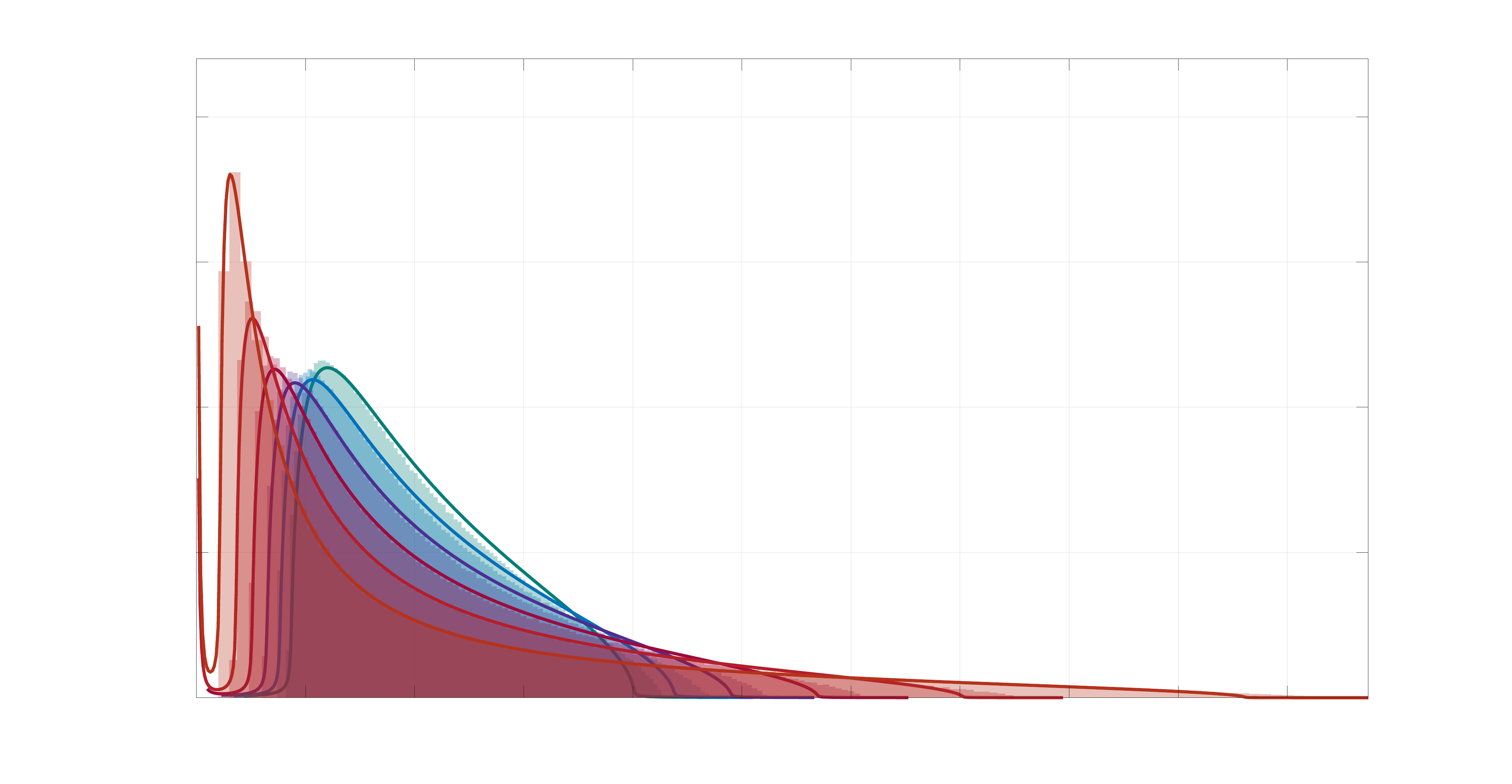}
        \caption{$\mu_1 = \delta_1$, $\Sigma^{(2)}_{ij} = \rho^{\vert i-j\vert}$ with $\rho=.9$, $\gamma\in[3,8].$}
        \label{fig:2}
    \end{subfigure}

    % Row 2
    \begin{subfigure}{0.45\textwidth}
        \centering\captionsetup{width=1.2\linewidth}
        \includegraphics[width=.8\linewidth]{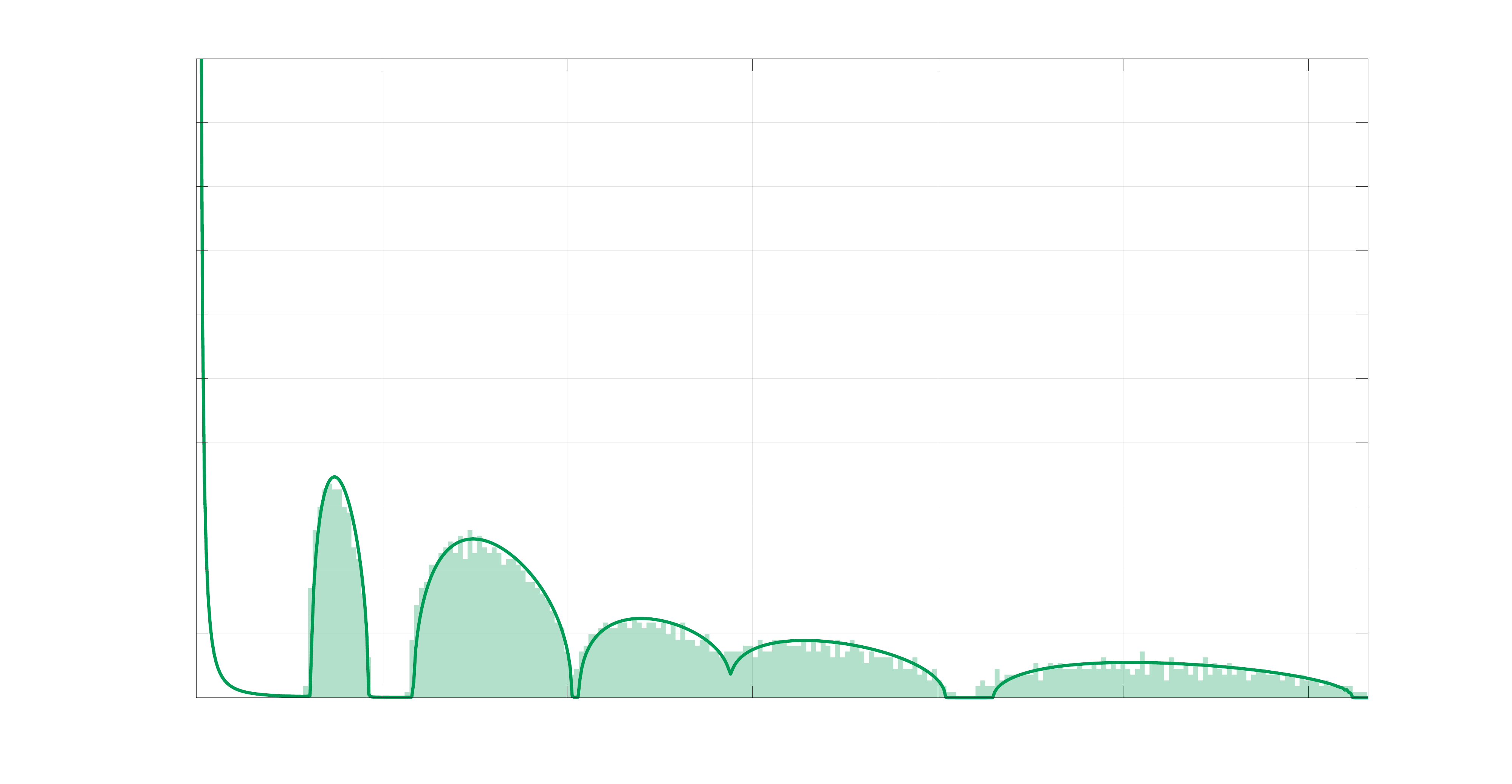}
        \caption{$\mu_1 = \frac{1}{2}\left(\delta_{1}+\delta_2\right),$ $\mu_2 = \frac{1}{3}\left(\delta_1+\delta_2+\delta_3\right),$ $\gamma=.25$.}
        \label{fig:3}
    \end{subfigure}
    \begin{subfigure}{0.45\textwidth}
        \centering\captionsetup{width=1.2\linewidth}
        \includegraphics[width=.8\linewidth]{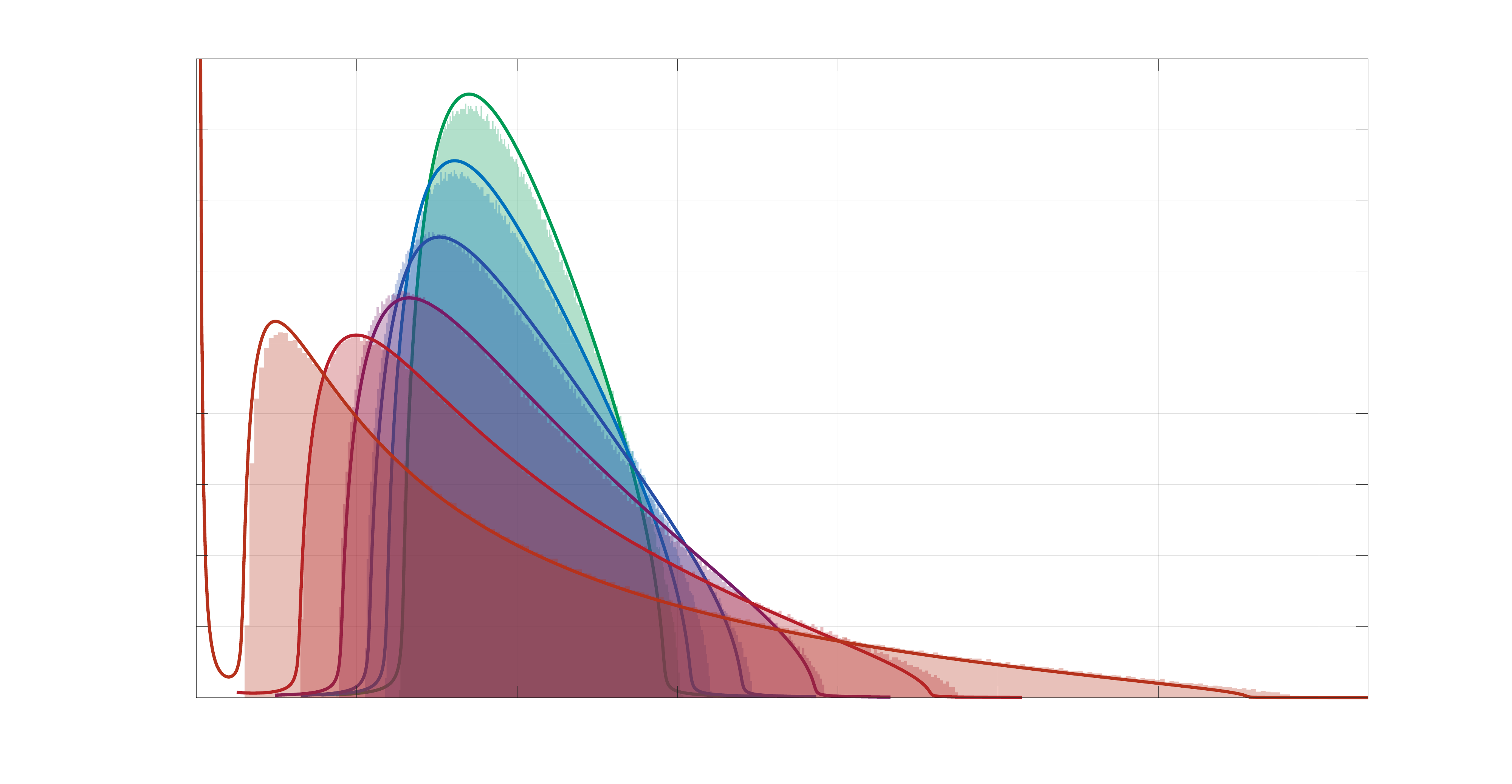}
        \caption{$\mu_1 = \mu_{\mathrm{MP}}^2,$ $\mu_2 = \mu_{\mathrm{MP}}^3,$ $\gamma\in[2,7]$}
        \label{fig:4}
    \end{subfigure}

    \caption{Histogram of eigenvalues of $M = \frac{1}{d_1}X^{(1)}{X^{(1)}}^\top \odot \frac{1}{d_2}X^{(2)}{X^{(2)}}^\top$ for specific covariances giving different $\mu_1$ and $\mu_2$. The curves are the theoretical prediction $\mu_{\mathrm{MP}}^{\gamma}\boxtimes (\mu_1\circledast \mu_2)$ for $n=35000,$ $d_1=d_2$ and different values of $\gamma$ specified for each figure.}
    \label{fig:simulation}
\end{figure}
\paragraph{Notation}
We denote $\db = (d_1,\dots,d_k)$ and $\vert \db\vert = \prod_{i=1}^k d_i$. We denote $[n]=\{1,\dots,n\}$ and $[\db] = \prod_{i=1}^k[d_i].$ We also denote $\R^\db = \bigotimes_{i=1}^k \Rbb^{d_i}$. We denote entries by lowercase letters $x_{ij}\in \Rbb$, vectors by lowercase bold letters $\x_i\in \Rbb^{d_i}$, matrices by uppercase letters $X\in \Rbb^{n\times d_i}$, and tensors by calligraphic letters tensors $\mathcal{X}\in \R^\db$. We also denote for a matrix (or tensor) $M:\Rbb^n\to\Rbb^n$,
\[
\mathrm{limspec}(M) = \lim_{n\to\infty}\frac{1}{n}\sum_{i=1}^n \delta_{\lambda_i}
\]
where $(\lambda_i)_i$ are the eigenvalues of $M$ and the limit is a weak convergence.
\section{Proof of the main result}
We start by noticing that if we consider an entry of $M$,
\[
m_{pq} = \frac{1}{\vert \db\vert}\prod_{i=1}^k \langle \x_p^{(i)},\x_q^{(i)}\rangle_{\Rbb^{d_i}}
=
\frac{1}{\vert \db\vert}\langle \W^p, \W^q\rangle_{\R^\db}
\quad\text{where we defined}\quad
\W^p = \bigotimes_{i=1}^k \x_p^{(i)}\in\R^\db.  
\]
Thus, for $(\mathbf{e}_p)_{p=1}^n$ the standard basis of $\Rbb^n$, we have 
\[
M = \frac{1}{\vert \db\vert} \A^\top \A \quad\text{with}\quad
\A = \sum_{p=1}^n \W^p\otimes \mathbf{e}_p \in \R^\db\otimes \Rbb^n.
\]
We note that the ``columns'' of $\A$, $\A_{\cdot p}$ are independent, and thus we are in the setting to use the theorem by Bai--Zhou \cite{baizhou}
\begin{theorem}[\cite{baizhou}]\label{theo:baizhou}
    For all $p\in[n]$, if we have $\E\left[\A_{\u p}\A_{\v p}\right]=\T_{\u\v}$ and for any non-random tensor $\B\in \R^\db\otimes \R^{\db}$ with bounded norm,
    \begin{equation}\label{eq:concquad}
    \frac{1}{n^2}\E\left[
        \left(
            \A_{\cdot p}^\top\B\A_{\cdot p}-\Tr(\B \T)
        \right)^2
    \right]
    \xrightarrow[n\to\infty]{}0.
    \end{equation}
    If additionally the norm of the tensor $\mathcal{T}$ is uniformly bounded and the empirical eigenvalue distribution of $\T$ converges to a non-random probability distribution $\nu$ then 
    \[
    \mathrm{limspec}\left(\frac{1}{\vert \db\vert}\A^\top\A\right) = \mu_{\mathrm{MP}}^\gamma\boxtimes \nu.
    \]
\end{theorem}
We thus see by this theorem that we need to first compute the covariance tensor $\T$, to prove that it is uniformly bounded and that its empirical eigenvalue distribution converges to some measure $\nu$. We then have to consider the concentration of quadratic form \eqref{eq:concquad}.
\paragraph{Computation of $\T$.} We start by simply computing the tensor $\T\in \R^\db\otimes \R^{\db}$.
\begin{lemma}\label{lem:covtensor}
    We have 
    \[
    \T
    =
    \pi\left(
        \bigotimes_{i=1}^k \Sigma^{(i)}
    \right)
    \]
    where $\pi$ is the following interlacing braid operator such that for $\u,\v\in [\db]$,
    \[
    \pi(\B)_{\u\v} = \pi(\B)_{u_1,\dots,u_k,v_1,\dots,v_k}
    =
    \B_{u_1,v_1,\dots,u_k,v_k}.
    \]
\end{lemma}
\begin{proof}
    We compute for $p\in[n]$,
    \[
    \E\left[
        \A_{\u p}\A_{\v p}
    \right]
    =
    \E\left[
        \prod_{i=1}^k 
        x^{(i)}_{pu_i}x^{(i)}_{pv_i}
    \right]
    =
    \prod_{i=1}^k
    \sigma^{(i)}_{u_iv_i}
    =
    \pi\left(
        \bigotimes_{i=1}^k \Sigma^{(i)}
    \right)_{\u\v}
    \]
    by definition of the interlacing braid operator.
\end{proof}
We can compute the limiting eigenvalue distribution of $\T$.
\begin{lemma}\label{lem:covtensor2}
    We have that $\T$ is uniformly bounded and that
    \[
    \nu \coloneqq \mathrm{limspec}\left(\T\right)
    =
    \mu_1\circledast\dots\circledast \mu_k.
    \]
\end{lemma}
\begin{proof}
     For $i\in[k]$, we denote $(\w^{(i)}_p)_{1\leqslant p\leqslant d_i}$ the eigenvectors of $\Sigma^{(i)}$ such that $\Sigma^{(i)}\w^{(i)}_p = \lambda_p^{(i)}\w^{(i)}_p$. 
    Then we see that for $p_i\in[d_i]$ and $\u\in[\db]$,
    \[
    \T\left(\bigotimes_{i=1}^k \w^{(i)}_{p_i}\right)_\u
    =
    \sum_{\v\in[\db]}
    \pi\left(\bigotimes_{i=1}^k \Sigma^{(i)}\right)_{\u\v} 
    \left(\bigotimes_{i=1}^k \w^{(i)}_{p_i}\right)_\v
    =
    \sum_{\v\in[\db]}
    \prod_{i=1}^k \sigma^{(i)}_{u_iv_i}w_{v_ip_i}^{(i)}
    =
    \prod_{i=1}^k
    \left(\Sigma^{(i)}\w_{p_i}\right)_{u_i}.
    \]
    Using the fact that $\w_p^{(i)}$ is an eigenvector of $\Sigma^{(i)}$ we see that 
    \[
    \T\left(\bigotimes_{i=1}^k \w^{(i)}_{p_i}\right)_\u
    =
    \prod_{i=1}^k \lambda_{p_i}^{(i)}w_{u_ip_i}^{(i)}
    =
    \prod_{i=1}^k \lambda_{p_i}^{(i)}
    \left(
        \bigotimes_{i=1}^k\w_{p_i}^{(i)}
    \right)_{\u}
    \]
    Thus we see that the spectrum of $\T$ is given by the family $(\prod_{i=1}^k\lambda_{p_i}^{(i)})_{\mathbf{p}\in [\db]}$ with eigenvectors $(\bigotimes_{i=1}^k\w_{p_i}^{(i)})_{\mathbf{p}\in[\db]}$. $\T$ being an operator from $\R^\db$ to $\R^\db$ this gives the whole spectrum of $\T$. In particular, since we have 
    \[
    \frac{1}{\vert \db\vert}\sum_{\mathbf{p}\in [\db]}\delta_{\prod_{i=1}^k\lambda_{p_i}^{(i)}}\xrightarrow[n\to\infty]{(d)}\mu_1\circledast\dots\circledast\mu_r,
    \]
    we obtain 
    \[
    \mathrm{limspec}(\T) = \mu_1\circledast\dots\circledast\mu_r.
    \]
    Since we suppose that $\Vert \Sigma^{(i)}\Vert \leqslant C$  for all $i\in[k]$, we see that $\Vert \T\Vert \leqslant C$.
\end{proof}
\paragraph{Concentration of quadratic forms.} The goal of this paragraph is to prove \eqref{eq:concquad} given by the following lemma 
\begin{proposition}\label{prop:concquad}
    Let $p\in[n]$ and $\B\in\R^\db\otimes\R^\db$ be a deterministic tensor with bounded norm, we have 
    \begin{equation}\label{eq:concquadbound}
    \frac{1}{n^2}\E\left[
        \left\vert 
            \A_{\cdot p}^\top\B\A_{\cdot p}-\Tr(\B\T)
        \right\vert^2
    \right]
    =
    \mathcal{O}_{k,C,\kappa_4,\Vert \B\Vert}\left(\frac{1}{d_{\min}}\right)
    \end{equation}
    where $d_{\min} = \min_{1\leqslant i\leqslant k} d_i$ and uniformly in $p\in[n]$.
\end{proposition}
Before proving this proposition, we start by writing the quantity considered in the following way 
    \[
     \frac{1}{n^2}\E\left[
        \left\vert 
            \A_{\cdot p}^\top\B\A_{\cdot p}-\Tr(\B\T)
        \right\vert^2
    \right]
    =
    \frac{1}{n^2}\E\left[
            \left(
                \A_{\cdot p}^\top \B\A_{\cdot p}
            \right)^2
    \right] 
    -
    \frac{1}{n^2}\Tr(\B\T)^2.
    \]
    Now, since only the $p$-th ``column'' of $\A$ is considered, we abuse our notation and identify $\A_{\cdot p}\leftrightarrow \A$ so that we want to compute
    \[
    \frac{1}{n^2}\E\left[
        \left(
            \A^\top \B\A
        \right)^2
    \right]
    =
    \frac{1}{n^2}\sum_{\u,\v,\u',\v'}\E\left[a_{\u}a_{\v}a_{\u'}a_{\v'}\right]b_{\u\v}b_{\u'\v'}.
    \]
    With this abuse of notation, note that $a_\u = a_{u_1,\dots,u_k,p}$. Using the definition of $\A$ and that the matrices $X^{(i)}$ are independent, if we still forget the index $p$ which is ubiquitous in all terms of this proof, we get 
    \begin{equation}\label{eq:quaddev}
    \frac{1}{n^2}\E\left[
        \left(
            \A^\top \B\A
        \right)^2
    \right]
    =
    \frac{1}{n^2}\sum_{\u,\v,\u',\v'}\prod_{i=1}^k\E\left[
        x_{u_i}^{(i)}x_{v_i}^{(i)}x_{u'_i}^{(i)}x_{v'_i}^{(i)}
    \right] 
    b_{\u\v}b_{\u'\v'}.
    \end{equation}
    We now use the cumulant expansion to compute these fourth moment,
    \[
    \E\left[
        x_{u}x_{v}x_{u'}x_{v'}
    \right]
    =
    \sum_{\pi}\prod_{B\in \pi}\kappa\left(
        \{x_{u},u\in B\}
    \right)
    \]
    where the sum is over all partition of the multiset $\{u,v,u',v'\}.$ Since the entries of $X^{(i)}$ are centered, we see that the terms involving the $(1,1,1,1)$, $(2,1,1),$ $(3,1)$ partitions are zero and only the $(2,2)$ and $(4)$  partitions give nonzero terms. Thus we can write 
    \[
    \E\left[x_ux_vx_{u'}x_{v'}\right]
    =
    \sigma_{uv}\sigma_{u'v'}+\sigma_{uu'}\sigma_{vv'}+\sigma_{uv'}\sigma_{u'v}+\kappa(x_u,x_v,x_{u'},x_{v'})
    \]
    In the rest of the article, we denote 
    \[
    \I^{(i)}_{uvu'v'} = \sigma_{uv}^{(i)}\sigma_{u'v'}^{(i)},
    \quad
    \II^{(i)}_{uvu'v'} = \sigma_{uu'}^{(i)}\sigma_{vv'}^{(i)},
    \quad 
    \III^{(i)}_{uvu'v'} = \sigma_{uv'}^{(i)}\sigma_{u'v}^{(i)}
    \quad\text{and} \quad
    \IV^{(i)}_{uvu'v'} = \kappa(x_u^{(i)},x_v^{(i)},x_{u'}^{(i)},x_{v'}^{(i)})
    \]
    so that 
    \[
    \E\left[
        x_ux_vx_{u'}x_{v'}
    \right]
    =
    \left(
        \I+\II+\III+\IV
    \right)_{uvu'v'}.
    \]
    We obtain Proposition \ref{prop:concquad} by induction over the number of matrices in the Hadamard product $k$. Indeed, we can see that if we choose one term in the cumulant expansion to be $\sigma_{uv}\sigma_{u'v'}$, we can reduce the problem to prove the concentration of quadratic forms for a tensor with one less index. For this reason, we give here a lemma detailing the case $k=1$, this is an easier object but we give the proof here for the sake of completeness. In particular, $T\in\Rbb^{d_1\times d_1}$ is a matrix which is equal to 
    \[
    T = \Sigma^{(1)}.
    \] 
    \begin{lemma}\label{lem:k=1}
        Let $p\in[n]$ and $B\in \Rbb^{d_1\times d_1}$ be a deterministic matrix with bounded norm, we have 
        \[
        \frac{1}{n^2}\E\left[
            \left\vert   
                (\x_p^{(1)})^\top B \x_p^{(1)}-\Tr\left(B\Sigma^{(1)}\right)
            \right\vert^2
        \right]
        \leqslant    
        \left(
            2C_1^2+\kappa_4^{(1)}
        \right)\Vert B\Vert^2 \frac{d_1}{n^2}.
        \]
    \end{lemma}
    \begin{proof}
        We start by doing the same process as above to write 
        \begin{multline*}
        \frac{1}{n^2}\E\left[
            \left\vert   
                (\x_p^{(1)})^\top B \x_p^{(1)}-\Tr\left(B\Sigma^{(1)}\right)
            \right\vert^2
        \right]
        \\=
        \frac{1}{n^2}\sum_{u,v,u',v'}
        \left(
            \sigma^{(1)}_{uv}\sigma^{(1)}_{u'v'}
            +
            \sigma^{(1)}_{uu'}\sigma^{(1)}_{vv'}+\sigma_{uv'}^{(1)}\sigma_{u'v}^{(1)}+\kappa(x_{pu}^{(1)},x_{pv}^{(1)},x_{pu'}^{(1)},x_{pv'}^{(1)})
        \right)b_{uv}b_{u'v'}-
        \frac{1}{n^2}\Tr(B\Sigma^{(1)})^2.
        \end{multline*}
        If we consider the first term, we obtain 
        \[
        \frac{1}{n^2}\sum_{u,v,u',v'}\sigma^{(1)}_{uv}b_{uv}\sigma^{(1)}_{u'v'}b_{uv}b_{u'v'}
        =
        \frac{1}{n^2}\left(
            \sum_{u,v=1}^{d_1} b_{uv}\sigma^{(1)}_{vu}
        \right)^2
        =
        \frac{1}{n^2}\Tr\left(B\Sigma^{(1)}\right)^2
        \]
        where we used the fact that $\Sigma^{(1)}$ is symmetric, this term thus cancels the centering. For the second term we have 
        \[
        \begin{aligned}
        \frac{1}{n^2}\left\vert\sum_{u,v,u',v'}\sigma_{uu'}^{(1)}\sigma_{vv'}^{(1)}b_{uv}b_{u'v'}\right\vert
        &=
        \frac{1}{n^2}\left\vert\sum_{u',v=1}^{d_1}
        \left(\Sigma^{(1)} B\right)_{u'v}\left(B\Sigma^{(1)}\right)_{u'v}\right\vert
        \\&\leqslant 
        \frac{1}{n^2}
        \left(
            \sum_{u',v=1}^{d_1}\left(\Sigma^{(1)} B\right)_{u'v}^2
        \right)^{\frac{1}{2}}
        \left(
            \sum_{u',v=1}^{d_1}\left(
                B\Sigma^{(1)}
            \right)
        \right)^{\frac{1}{2}}
        \\&=
        \frac{1}{n^2}\Vert \Sigma^{(1)}B\Vert_{\mathrm{F}}\Vert B\Sigma^{(1)}\Vert_{\mathrm{F}}
        \leqslant    
        \frac{d_1\Vert \Sigma^{(1)}\Vert^2\Vert B\Vert^2}{n^2}
        \leqslant C_1^2\Vert B\Vert^2\frac{d_1}{n^2}.
        \end{aligned}
        \]
        Similary, for the third term 
        \[
        \begin{aligned}
            \frac{1}{n^2}\left\vert\sum_{u,v,u',v'}\sigma^{(1)}_{uv'}\sigma^{(1)}_{u'v}b_{uv}b_{u'v'}\right\vert
            =
            \frac{1}{n^2}\left\vert\sum_{v,v'=1}^{d_1}\left(
                \Sigma^{(1)}B
            \right)_{vv'}\left(
                \Sigma^{(1)}B
            \right)_{v'v}\right\vert
            \end{aligned}
        \leqslant \frac{1}{n^2}\Vert\Sigma^{(1)}B\Vert_{\mathrm{F}}^2
        \leqslant C_1^2\Vert B\Vert^2\frac{d_1}{n^2}.
        \]
        Finally, for the term involving the fourth cumulant we can write, if we denote $\kappa_{uvu'v'}^{(1)}\coloneqq \kappa(x_{pu}^{(1)},x_{pv}^{(1)},x_{pu'}^{(1)},x_{pv'}^{(1)}),$ 
        \[
        \frac{1}{n^2}\left\vert\sum_{u,v,u',v'}\kappa_{uvu'v'}^{(1)}b_{uv}b_{u'v'}\right\vert    
        \leqslant
        \left(\sup_{1\leqslant u,v\leqslant d_1}\sum_{u'v'=1}^{d_1}\vert \kappa_{uvu'v'}^{(1)}\vert\right)
         \Vert B\Vert_{\mathrm{F}}^2\frac{1}{n^2}
         \leqslant \kappa_4^{(1)}\Vert B\Vert^2 \frac{d_1}{n^2}
        \]
        where we used the fact that $\vert \langle B,\mathcal{K}^{(1)}B\rangle\vert \leqslant \Vert B\Vert_F^2 \Vert \mathcal{K}^{(1)}\Vert_{\mathrm{F}\to\mathrm{F}}$ and the Schur test which gives 
        \[
        \Vert \mathcal{K}^{(1)}\Vert_{\mathrm{F}\to\mathrm{F}}
        \leqslant 
        \sqrt{
            \left(\sup_{1\leqslant u,v\leqslant d_1}\sum_{u',v'=1}^{d_1}\vert \kappa_{uvu'v'}^{(1)}\vert\right)\left(\sup_{1\leqslant u',v'\leqslant d_1}\sum_{u,v=1}^{d_1}\vert \kappa_{uvu'v'}^{(1)}\vert\right)
        } 
        = 
        \sup_{1\leqslant u,v\leqslant d_1}\sum_{u'v'=1}^{d_1}\vert \kappa_{uvu'v'}^{(1)}\vert 
        \] by symmetry of $\mathcal{K}^{(1)}$.
    \end{proof}
    The following lemma states that if one chooses the term $\I^{(i)}$ in the product in \eqref{eq:quaddev}, it corresponds to the same bound \eqref{eq:concquadbound} for a lower dimensional problem. Thus, with the initialization step done by the previous lemma in the one-dimensional case, this controls every term involving at least one $\I^{(i)}$.
    \begin{lemma}\label{lem:I}
        In the expansion of \[\prod_{i=1}^k \left(\I^{(i)}+\II^{(i)}+\III^{(i)}+\IV^{(i)}\right)_{\u\v\u'\v'}\] that appears in the quantity \eqref{eq:quaddev}, any monomial that contains $\I^{(j)}$ for at least one index $j$ reduces the bound \eqref{eq:concquadbound} to the $k-1$-tensor case i.e. the same expression but on $\bigotimes_{i\neq j}\Rbb^{d_i}$. \\
        Thus, if the $k-1$-tensor version of the bound \eqref{eq:concquadbound} holds then each monomial in the expansion that contains at least one $\I^{(j)}$ contributes at most this bound. 
    \end{lemma}
    \begin{proof}
    Without loss of generality, we suppose that $\I^{(1)}$ is chosen in the first term of the product in \eqref{eq:quaddev}, then we want to bound 
    \begin{multline}\label{eq:ind1}
    \frac{1}{n^2}\sum_{\substack{u_1,v_1,u'_1,v'_1\\ \u,\v,\u',\v'}} \I^{(1)}_{u_1v_1u'_1v'_1}\prod_{i=2}^{k}\E\left[
        x^{(i)}_{u_i}x_{v_i}^{(i)}x_{u'_i}^{(i)}x_{v'_i}^{(i)}
    \right]
    b_{u_1\u v_1\v}b_{u'_1\u' v'_1\v'}
    \\[-2ex]=
    \frac{1}{n^2}\sum_{\substack{u_1,v_1,u'_1,v'_1\\ \u,\v,\u',\v'}} \sigma^{(1)}_{u_1v_1}\sigma^{(1)}_{u'_1v'_1}\prod_{i=2}^{k}\E\left[
        x^{(i)}_{u_i}x_{v_i}^{(i)}x_{u'_i}^{(i)}x_{v'_i}^{(i)}
    \right]
    b_{u_1\u v_1\v}b_{u'_1\u' v'_1\v'}
    \end{multline}
    where in the sum we have $\u,\v,\u',\v'\in \prod_{i=2}^k \Rbb^{d_i}$. Define 
    \[
    m_{\u\v} = \frac{1}{d_1}\sum_{u_1,v_1=1}^{d_1} 
    \sigma^{(1)}_{u_1v_1}b_{u_1\u v_1\v}
    \]
    so that we can rewrite \eqref{eq:ind1} as 
    \[
    \eqref{eq:ind1} = \frac{d_1^2}{n^2}\sum_{\u,\v,\u',\v'}
    \prod_{i=2}^k \E\left[
        x^{(i)}_{u_i}x_{v_i}^{(i)}x_{u'_i}^{(i)}x_{v'_i}^{(i)}
    \right]
    m_{\u\v}m_{\u'\v'}
    =
    \frac{1}{n_1^2}\sum_{\u,\v,\u',\v'}
    \prod_{i=2}^k \E\left[
        x^{(i)}_{u_i}x_{v_i}^{(i)}x_{u'_i}^{(i)}x_{v'_i}^{(i)}
    \right]
    m_{\u\v}m_{\u'\v'}
    \]
    with $n_1\coloneqq \frac{n}{d_1}$. In particular we have that $\frac{n_1}{\prod_{i=2}^k d_i}\to \gamma$. Hence, we are precisely in the setting where we must bound the Hadamard product of $k-1$ independent random matrices, assuming that $\Vert \M\Vert$ is bounded. We now proceed to show this. Let $\F,\G\in \bigotimes_{i=2}^k \Rbb^{d_i}$ be two tensors such that $\Vert \F\Vert_{\mathrm{F}} = \Vert \G\Vert_{\mathrm{F}}=1$ then 
    \[
    \langle \F,\M\G\rangle
    =
    \sum_{\u,\v} f_\u g_\v m_{\u\v} = \frac{1}{d_1}\sum_{\substack{u_1,v_1\\ \u,\v}}f_\u g_\v\sigma^{(1)}_{u_1v_1}b_{u_1\u v_1\v}.
    \]
    If we define 
    \[
    \G^{(1)} = \Sigma^{(1)}_{u_1\cdot}\otimes \G\in \bigotimes_{i=1}^k \Rbb^{d_i} \quad\text{or in other words}\quad 
    g^{(1)}_{v_1 \v} = \sigma^{(1)}_{u_1v_1}g_{\v}
    \]
    then we can write 
    \[
    \langle \F,\M\G\rangle 
    =
    \frac{1}{d_1}\sum_{u_1,\u}f_\u \left[\B\G^{(1)}\right]_{u_1\u}  
    =
    \frac{1}{d_1}\sum_{u_1=1}^{d_1}
        \left\langle \left(
            \B\G^{(1)}
        \right)_{u_1\cdot},
        \F
        \right\rangle.
    \]
    But we have 
    \[
    \left\vert   
        \left\langle \left(
            \B\G^{(1)}
        \right)_{u_1\cdot},
        \F
        \right\rangle
    \right\vert
    \leqslant \left\Vert \B\G^{(1)}\right\Vert_{\mathrm{F}}\Vert \F\Vert_{\mathrm{F}}
    \leqslant \Vert \B\Vert \left\Vert \G^{(1)}\right\Vert_{\mathrm{F}}
    \]
    using the fact that $\Vert\F\Vert_{\mathrm{F}}=1$. Now we have that 
    \[
    \left\Vert \G^{(1)}\right\Vert_{\mathrm{F}}^2   = \sum_{v_1\v}\sigma_{u_1v_1}^2g_\v^2 = \Vert \G\Vert_{\mathrm{F}}^2 \sum_{v_1=1}^{d_1}\sigma_{u_1v_1}^2 \leqslant \Vert \Sigma^{(1)}\Vert^2.
    \]
    Finally, we obtain the bound 
    \[
    \Vert \M\Vert \leqslant \Vert \B\Vert \Vert\Sigma^{(1)}\Vert.
    \]
    \end{proof}
    
    Next, we consider the case where only $\II$ and $\III$ occur in the product in \eqref{eq:quaddev}.
    
\begin{lemma}\label{lem:II-III}
   We have
    \[
        \frac{1}{n^2}\sum_{\u,\v,\u',\v'}
        \left(
            \prod_{i=1}^k \II^{(i)}_{}+\III^{(i)}
        \right)_{\u\v\u'\v'}b_{\u\v}b_{\u'\v'}\ 
        \leqslant 
        2^kC^{2k}\Vert \B\Vert^2\frac{\vert \db\vert^{\frac{3}{2}}}{n^2}
    \]
\end{lemma}
     
\begin{proof}
    Without loss of generality, we consider the following form
    \begin{equation}\label{eq:form}
     \frac{1}{n^2}\sum_{\u,\v,\u',\v'}
     \left(
        \prod_{i=1}^m \II^{(i)}_{u_iv_iu'_iv'_i} \prod_{j=m+1}^{k}\III^{(j)}_{u_jv_ju'_jv'_j}
    \right)b_{\u\v}b_{\u'\v'}
    \end{equation}
    as the other forms follow by symmetry. Note that the case $m=0$ consists in only $\III^{(i)}$ terms and $m=k$ consists in only $\II^{(i)}$ terms.  By definition, we can rewrite \eqref{eq:form}
    as 
    \[
    \eqref{eq:form}
    =
    \frac{1}{n^2}\sum_{\u,\v,\u',\v'}\prod_{i=1}^m \sigma_{u_iu_i'}^{(i)}\sigma_{v_iv_i'}^{(i)} \prod_{j=m+1}^{k}\sigma_{u_jv_j'}^{(j)}\sigma_{u_j'v_j}^{(j)}b_{\u\v}b_{\u'\v'}
    \]
    and let us denote for $m\in\{0,\dots, k\}$,
    \[\A_{\u_{1:m},\u_{m+1:k}}^{\u'_{1:m},\v'_{m+1:k}} = \prod_{i=1}^m \sigma_{u_iu_i'}^{(i)}\prod_{j=m+1}^k\sigma_{u_jv'_j}^{(j)}
    \quad
    \text{and} \quad\A_{\v_{1:m},\v_{m+1:k}}^{\v'_{1:m},\u'_{m+1:k}} = \prod_{i=1}^m \sigma_{v_iv_i'}^{(i)}\prod_{j=m+1}^k\sigma_{v_ju'_j}^{(j)}
    \]
     where we introduced the notation $\u_{i:j}$ = $(u_i,...,u_j)$,
    $\v_{i:j}$ = $(v_i,...,v_j)$
    and $\u_{i:j},\v_{l:m} = (u_i,...,u_j,v_l...v_m)$
    then \eqref{eq:form} becomes 
    \[
    \begin{aligned}
     \eqref{eq:form} 
    =
    \frac{1}{n^2}\sum_{\u',\v'}
    \left\langle
        \A^{\u'_{1:m},\v'_{m+1:k}},\B \A^{\v'_{1:m},\u'_{m+1:k}}
    \right\rangle b_{\u'\v'}
    &\leq  \frac{1}{n^2} \Vert\B\Vert_{\mathrm{F}}
    \sqrt{
        \sum_{\u',\v'}\left\langle
            \,\A^{\u'_{1:m},\v'_{m+1:k}},\B \A^{\v'_{1:m},\u'_{m+1:k}}
        \right\rangle^2
    }
    \\
    &\leq \frac{1}{n^2} \Vert\B\Vert_{\mathrm{F}}
    \sqrt{\sum_{\u',\v'}\Vert\B\Vert^2
     \Vert\A^{\u'_{1:m},\v'_{m+1:k}}\Vert_{\mathrm{F}}^2\Vert\A^{\v'_{1:m},\u'_{m+1:k}}\Vert_{\mathrm{F}}^2}
     \\
      &\leq \frac{\Vert \B\Vert \Vert\B\Vert_{\mathrm{F}}}{n^2} \sqrt{\prod_{i=1}^k\Vert\Sigma^{(i)}\Vert_{\mathrm{F}}^4}
     \\
      &\leq \frac{\sqrt{\prod_{i=1}^k d_i}}{n^2} \Vert\B\Vert^2\prod_{i=1}^k d_i
      \prod_{i=1}^k\Vert \Sigma^{(i)}\Vert^2
      \leqslant 
      C^{2k}\Vert \B\Vert^2\frac{\vert \db\vert^{\frac{3}{2}}}{n^2}
      \end{aligned}
      \]  
   where we used the fact that
     $\Vert\Sigma^{(i)}\Vert_{\mathrm{F}} \leq \sqrt{d_i} \Vert\Sigma^{(i)}\Vert$ and $\Vert\B\Vert_{\mathrm{F}}\leqslant \sqrt{\vert \db\vert}\Vert \B\Vert.$ We obtain the final bound on noticing that there are $2^k$ terms in the expansion.

\end{proof}

We now extend the previous lemma by adding IV terms in the product in \eqref{eq:quaddev}
\begin{lemma}\label{lem:IV}
    we have
    \begin{multline*}
        \frac{1}{n^2}\sum_{\u,\v,\u',\v'}
        \left(
            \prod_{i=1}^k \II^{(i)}_{}+\III^{(i)}+\IV^{(i)}
        \right)_{\u\v\u'\v'}b_{\u\v}b_{\u'\v'}\ 
        \\\leqslant 
        2^kC^{2k}\Vert \B\Vert^2\frac{\vert \db\vert^{\frac{3}{2}}}{n^2}
        +
        (3^k-2^k)\kappa_4^{k}\Vert \B\Vert^2
        \sup_{1\leqslant \ell\leqslant k-1}
        \left(\frac{C^2}{\kappa_4}\right)^{\ell}
        \sup_{1\leqslant i_1,\dots,i_\ell\leqslant k}
        \frac{\vert \db\vert \prod_{j=1}^\ell d_{i_j}}{n^2}
        \end{multline*}
\end{lemma}

\begin{proof}
    Without loss of generality, we consider the following form
    \begin{equation}\label{eq:form1}
     \frac{1}{n^2}\sum_{\u,\v,\u',\v'}
     \left(
        \prod_{i=1}^m \II^{(i)}_{u_iv_iu'_iv'_i} \prod_{j=m+1}^{\ell}\III^{(j)}_{u_jv_ju'_jv'_j}
        \prod_{q=\ell+1}^{k}\IV^{(q)}_{u_{q}v_{q}u'_{q}v'_{q}}
    \right)b_{\u\v}b_{\u'\v'}
    \end{equation}
     as the other forms follow by symmetry. Note that the case where no $\IV$ terms are chosen was treated in Lemma \ref{lem:II-III}, and we therefore assume that at least one term is chosen to be a $\IV$ term. Hence the case $m=0$ consists in only $\III^{(i)}$ and $\IV^{(i)}$ terms and $m = \ell$ consists in only $\II^{(i)}$ and $\IV^{(i)}$ terms.
      \\  By definition, we have 
    \[
    \eqref{eq:form1}
    =
    \frac{1}{n^2}\sum_{\u,\v,\u',\v'}\prod_{i=1}^m \sigma_{u_iu_i'}^{(i)}\sigma_{v_iv_i'}^{(i)} \prod_{j=m+1}^{\ell}\sigma_{u_jv_j'}^{(j)}\sigma_{u_j'v_j}^{(j)}
    \prod_{q=l+1}^{k}
    \kappa\left(x_{u_q}^{(q)},x_{v_q}^{(q)},x_{u'_q}^{(q)},x_{v'_q}^{(q)}\right)
    b_{\u\v}b_{\u'\v'}
    \]
    if we denote 
    \[
    \mathcal{K}_{\u\v\u'\v'} = \prod_{i=1}^m \sigma_{u_iu_i'}^{(i)}\sigma_{v_iv_i'}^{(i)} \prod_{j=m+1}^{\ell}\sigma_{u_jv_j'}^{(j)}\sigma_{u_j'v_j}^{(j)}
    \prod_{q=\ell+1}^{k}
    \kappa\left(x_{u_q}^{(q)},x_{v_q}^{(q)},x_{u'_q}^{(q)},x_{v'_q}^{(q)}\right)
    \]
    then we can write $\vert \eqref{eq:form1} \vert$
    as 
     \begin{equation}\label{eq:form2}
     \vert \eqref{eq:form1} \vert 
    =
    \frac{1}{n^2} 
    \vert \langle \B,\mathcal{K}\B\rangle\vert \leqslant
    \frac{1}{n^2} \Vert \B\Vert_F^2 \Vert \mathcal{K}\Vert_{\mathrm{F}\to\mathrm{F}} 
    \end{equation}
    using Schur test, and exploiting the symmetry of $\mathcal{K}$
    \[
    \begin{aligned}
     \Vert \mathcal{K}\Vert_{\mathrm{F}\to\mathrm{F}}
    &\leqslant 
        \sqrt{
            \left(\sup_{\u,\v}\sum_{\u',\v'}^{}\vert \mathcal{K}_{\u\v\u'\v'}\vert\right)\left(\sup_{ \u',\v'}\sum_{\u,\v}^{}\vert \mathcal{K}_{\u\v\u'\v'}\vert\right)
        } 
        = \sup_{\u,\v}\sum_{\u',\v'}^{}\vert \mathcal{K}_{\u\v\u'\v'}\vert
    \\
    &\leqslant \kappa_4^{k-\ell}
    \sup_{\u_{1:\ell},\v_{1:\ell}}\sum_{\u_{1:\ell}',\v_{1:\ell}'}^{}
    \left\vert
        \prod_{i=1}^m \sigma_{u_iu_i'}^{(i)}\sigma_{v_iv_i'}^{(i)} \prod_{j=m+1}^{l}\sigma_{u_jv_j'}^{(j)}\sigma_{u_j'v_j}^{(j)}
    \right\vert \\
    &\leqslant  \kappa_4^{k-\ell} C^{2\ell}
    \prod_{i=1}^{\ell} d_i 
    \end{aligned}
    \]
    where we used the fact that 
    \[
     \forall 1 \leq i \leq k 
     \sup_{1\leq\u_{i},\v_{i}\leq   d_i}\sum_{\u_i',\v_i' =  1}^{d_i}\left\vert  \sigma_{u_iu_i'}^{(i)}\sigma_{v_iv_i'}^{(i)}
    \right\vert \quad , \quad  \sup     _{1\leq\u_{i},\v_{i}\leq d_i}\sum_{\u_i',\v_i' = 1}^{d_i}\left\vert            \sigma_{u_iv_i'}^{(i)}\sigma_{u_i'v_i}^{(i)}
    \right\vert \quad\leq d_i \Vert \Sigma^{(i)}\Vert^2  \
    \]
    then we have
    \[
    \begin{aligned}
      \eqref{eq:form2}  \leqslant 
      \frac{\kappa_4^{k-\ell} C^{2\ell}}{n^2} \Vert \B\Vert_F^2 
      \prod_{i=1}^{\ell} d_i 
      \leqslant \kappa_4^{k-\ell}C^{2\ell}\Vert \B\Vert^2\frac{\vert \db\vert \prod_{i=1}^\ell d_i}{n^2}.
    \end{aligned}
    \]
    We obtain the final bound by setting optimzing in the subset of $d_i$ chosen and noticing that there are $3^k - 2^k$ terms in the expansion.
\end{proof}
We are now ready to prove Proposition \ref{prop:concquad}.
\begin{proof}[Proof of Proposition \ref{prop:concquad}]
    We simply have to combine Lemmas \ref{lem:k=1}, \ref{lem:I} ,\ref{lem:II-III}, and \ref{lem:IV}. Lemmas \ref{lem:I} and \ref{lem:k=1} handle all terms involving a $\I$ term. Lemma \ref{lem:II-III} handle all combinations of $\II$ and $\III$ terms while Lemmas \ref{lem:IV} handle all terms involving a $\IV$ term. We just have to consider the largest error from all these lemmas and we see that it is $\O{\frac{1}{\sqrt{\vert \db\vert}}+\frac{1}{d_{\min}}}$ where $d_{\min} = \min_{1\leqslant i\leqslant k} d_i$. As soon as $k\geqslant 2$, we see that the largest of the two terms is $\O{\frac{1}{d_{\min}}}$ and the case $k=1$ is given by Lemma \ref{lem:k=1} which gives a $\O{\frac{1}{d_1}}=\O{\frac{1}{d_{\min}}}$ since there is only one term. We thus obtain the final result.
\end{proof}
We now are able to prove our main theorem.
\begin{proof}[Proof of Theorem \ref{theo:main}]
    This is now an application of Theorem \ref{theo:baizhou}. By Lemma \ref{lem:covtensor2}, the covariance tensor is uniformly bounded and its limiting spectrum is given by $\mu_1 \circledast \dots\circledast \mu_k$ and Proposition \ref{prop:concquad} gives us the concentration of quadratic forms to apply Theorem \ref{theo:baizhou}. Finally we obtain that 
    \[
    \mathrm{limspec}\left(  
        \bigodot_{i=1}^k \frac{1}{d_i}X^{(i)}{X^{(i)}}^\top
    \right)
    =
    \mu_{\mathrm{MP}}^{\gamma} \boxtimes \left(
    \mu_1\circledast \dots\circledast \mu_k
    \right).
    \]
\end{proof}
%\appendix 
%\appendixswitch
%\renewcommand{\thesection}{\Alph{section}}
%\section{Proof of Lemma \ref{lem:tightness}}\label{app:prooftight}
\bibliographystyle{abbrv}
\bibliography{bibliostage}

% \bib, bibdiv, biblist are defined by the amsrefs package.
\begin{bibdiv}
\begin{biblist}

\bib{abou2025eigenvalue}{article}{
      author={Abou~Assaly, S.},
      author={Benigni, L.},
       title={Eigenvalue distribution of the hadamard product of sample
  covariance matrices in a quadratic regime},
        date={2025},
     journal={arXiv e-prints},
       pages={arXiv\ndash 2502},
}

\bib{bose2014bulk}{article}{
      author={Bose, A.},
      author={Mukherjee, S.~S.},
       title={Bulk behavior of {S}chur--{H}adamard products of symmetric random
  matrices},
        date={2014},
     journal={Random Matrices: Theory and Applications},
      volume={3},
      number={02},
       pages={1450007},
}

\bib{benignipeche1}{article}{
      author={Benigni, L.},
      author={P\'ech\'e, S.},
       title={Eigenvalue distribution of some nonlinear models of random
  matrices},
        date={2021},
        ISSN={1083-6489},
     journal={Electron. J. Probab.},
      volume={26},
       pages={Paper No. 150, 37},
         url={https://doi.org/10.1214/21-ejp699},
}

\bib{benignipeche2}{article}{
      author={Benigni, L.},
      author={P{\'e}ch{\'e}, S.},
       title={Largest eigenvalues of the conjugate kernel of single-layered
  neural networks},
        date={2022},
     journal={arXiv preprint},
}

\bib{benignipaquette}{article}{
      author={Benigni, L.},
      author={Paquette, E.},
       title={Eigenvalue distribution of the neural tangent kernel in the
  quadratic scaling},
        date={2025},
     journal={arXiv preprint},
}

\bib{baisilver}{book}{
      author={Bai, Z.},
      author={Silverstein, J.~W.},
       title={Spectral analysis of large dimensional random matrices},
     edition={Second},
      series={Springer Series in Statistics},
   publisher={Springer, New York},
        date={2010},
        ISBN={978-1-4419-0660-1},
}

\bib{baizhou}{article}{
      author={Bai, Z.},
      author={Zhou, W.},
       title={Large sample covariance matrices without independence structures
  in columns},
        date={2008},
     journal={Statist. Sinica},
      volume={18},
      number={2},
       pages={425\ndash 442},
}

\bib{rmmlbook}{book}{
      author={Couillet, R.},
      author={Liao, Z.},
       title={Random matrix methods for machine learning},
   publisher={Cambridge University Press},
        date={2022},
}

\bib{chizat2019lazy}{article}{
      author={Chizat, L.},
      author={Oyallon, E.},
      author={Bach, F.},
       title={On lazy training in differentiable programming},
        date={2019},
     journal={Advances in neural information processing systems},
      volume={32},
}

\bib{chengsinger}{article}{
      author={Cheng, X.},
      author={Singer, A.},
       title={The spectrum of random inner-product kernel matrices},
        date={2013},
        ISSN={2010-3263,2010-3271},
     journal={Random Matrices Theory Appl.},
      volume={2},
      number={4},
       pages={1350010, 47},
}

\bib{dubova2023universality}{article}{
      author={Dubova, S.},
      author={Lu, Y.~M},
      author={McKenna, B.},
      author={Yau, H.-T.},
       title={Universality for the global spectrum of random inner-product
  kernel matrices in the polynomial regime},
        date={2023},
     journal={arXiv preprint},
}

\bib{dabo2024traffic}{article}{
      author={Dabo, I.},
      author={Male, C.},
       title={A traffic approach for profiled pennington-worah matrices},
        date={2024},
     journal={arXiv preprint},
}

\bib{elkaroui}{article}{
      author={El~Karoui, N.},
       title={The spectrum of kernel random matrices},
        date={2010},
        ISSN={0090-5364,2168-8966},
     journal={Ann. Statist.},
      volume={38},
      number={1},
       pages={1\ndash 50},
}

\bib{fan2020spectra}{article}{
      author={Fan, Z.},
      author={Wang, Z.},
       title={Spectra of the conjugate kernel and neural tangent kernel for
  linear-width neural networks},
        date={2020},
     journal={Advances in neural information processing systems},
      volume={33},
       pages={7710\ndash 7721},
}

\bib{hu2024asymptotics}{article}{
      author={Hu, H.},
      author={Lu, Y.~M},
      author={Misiakiewicz, T.},
       title={Asymptotics of random feature regression beyond the linear
  scaling regime},
        date={2024},
     journal={arXiv preprint},
}

\bib{hachem2007deterministic}{article}{
      author={Hachem, W.},
      author={Loubaton, P.},
      author={Najim, J.},
       title={Deterministic equivalents for certain functionals of large random
  matrices},
        date={2007},
        ISSN={1050-5164,2168-8737},
     journal={Ann. Appl. Probab.},
      volume={17},
      number={3},
       pages={875\ndash 930},
}

\bib{jacot2018neural}{article}{
      author={Jacot, A.},
      author={Gabriel, F.},
      author={Hongler, C.},
       title={Neural tangent kernel: Convergence and generalization in neural
  networks},
        date={2018},
     journal={Advances in neural information processing systems},
      volume={31},
}

\bib{kogan2024extremal}{article}{
      author={Kogan, D.},
      author={Nandy, S.},
      author={Huang, J.},
       title={Extremal eigenvalues of random kernel matrices with polynomial
  scaling},
        date={2024},
     journal={arXiv preprint arXiv:2410.17515},
}

\bib{louartliaocouillet}{article}{
      author={Louart, C.},
      author={Liao, Z.},
      author={Couillet, R.},
       title={A random matrix approach to neural networks},
        date={2018},
        ISSN={1050-5164,2168-8737},
     journal={Ann. Appl. Probab.},
      volume={28},
      number={2},
       pages={1190\ndash 1248},
}

\bib{lu2022equivalence}{article}{
      author={Lu, Y.~M},
      author={Yau, H.-T.},
       title={An equivalence principle for the spectrum of random inner-product
  kernel matrices with polynomial scalings},
        date={2022},
     journal={arXiv preprint},
}

\bib{misiakiewicz2022spectrum}{article}{
      author={Misiakiewicz, T.},
       title={Spectrum of inner-product kernel matrices in the polynomial
  regime and multiple descent phenomenon in kernel ridge regression},
        date={2022},
     journal={arXiv preprint},
}

\bib{misiakiewicz2023six}{article}{
      author={Misiakiewicz, T.},
      author={Montanari, A.},
       title={Six lectures on linearized neural networks},
        date={2023},
     journal={arXiv preprint},
}

\bib{marchenko}{article}{
      author={Marchenko, V.~A.},
      author={Pastur, L.~A.},
       title={Distribution of eigenvalues for some sets of random matrices},
        date={1967},
     journal={Mathematics of the USSR-Sbornik},
      volume={1},
      number={4},
       pages={457},
}

\bib{mukherjee2023convergence}{article}{
      author={Mukherjee, S.~S.},
       title={On *-convergence of {S}chur--{H}adamard products of independent
  nonsymmetric random matrices},
        date={2023},
     journal={International Mathematics Research Notices},
      volume={2023},
      number={17},
       pages={14667\ndash 14698},
}

\bib{montanarizhong}{article}{
      author={Montanari, A.},
      author={Zhong, Y.},
       title={The interpolation phase transition in neural networks:
  memorization and generalization under lazy training},
        date={2022},
        ISSN={0090-5364,2168-8966},
     journal={Ann. Statist.},
      volume={50},
      number={5},
       pages={2816\ndash 2847},
         url={https://doi.org/10.1214/22-aos2211},
}

\bib{piccolo2021analysis}{article}{
      author={Piccolo, V.},
      author={Schr{\"o}der, D.},
       title={Analysis of one-hidden-layer neural networks via the resolvent
  method},
        date={2021},
     journal={Advances in Neural Information Processing Systems},
      volume={34},
       pages={5225\ndash 5235},
}

\bib{pennington2017nonlinear}{article}{
      author={Pennington, J.},
      author={Worah, P.},
       title={Nonlinear random matrix theory for deep learning},
        date={2017},
     journal={Advances in neural information processing systems},
      volume={30},
}

\bib{pandit2024universality}{article}{
      author={Pandit, P.},
      author={Wang, Z.},
      author={Zhu, Y.},
       title={Universality of kernel random matrices and kernel regression in
  the quadratic regime},
        date={2024},
     journal={arXiv preprint},
}

\bib{silverstein1995empirical}{article}{
      author={Silverstein, J.~W.},
      author={Bai, Z.~D.},
       title={On the empirical distribution of eigenvalues of a class of
  large-dimensional random matrices},
        date={1995},
        ISSN={0047-259X,1095-7243},
     journal={J. Multivariate Anal.},
      volume={54},
      number={2},
       pages={175\ndash 192},
}

\bib{silverstein1995strong}{article}{
      author={Silverstein, J.~W.},
       title={Strong convergence of the empirical distribution of eigenvalues
  of large-dimensional random matrices},
        date={1995},
        ISSN={0047-259X,1095-7243},
     journal={J. Multivariate Anal.},
      volume={55},
      number={2},
       pages={331\ndash 339},
}

\bib{wigner}{article}{
      author={Wigner, E.~P.},
       title={Characteristic vectors of bordered matrices with infinite
  dimensions},
        date={1955},
        ISSN={0003-486X},
     journal={Ann. of Math. (2)},
      volume={62},
       pages={548\ndash 564},
}

\bib{wishart1928generalised}{article}{
      author={Wishart, J.},
       title={The generalised product moment distribution in samples from a
  normal multivariate population},
        date={1928},
     journal={Biometrika},
       pages={32\ndash 52},
}

\bib{wangzhu}{article}{
      author={Wang, Z.},
      author={Zhu, Y.},
       title={Deformed semicircle law and concentration of nonlinear random
  matrices for ultra-wide neural networks},
        date={2024},
        ISSN={1050-5164,2168-8737},
     journal={Ann. Appl. Probab.},
      volume={34},
      number={2},
       pages={1896\ndash 1947},
}

\bib{yin1986limiting}{article}{
      author={Yin, Y.~Q.},
       title={Limiting spectral distribution for a class of random matrices},
        date={1986},
        ISSN={0047-259X},
     journal={J. Multivariate Anal.},
      volume={20},
      number={1},
       pages={50\ndash 68},
}

\end{biblist}
\end{bibdiv}
\end{document}